\documentclass[12pt]{article}
\usepackage{amsmath,amsfonts,amsthm,bbm}

\usepackage[a4paper]{geometry}
\newtheorem{theorem}{Theorem}
\newtheorem{lemma}[theorem]{Lemma}

\newcommand{\R}{\mathbb R}
\newcommand{\N}{\mathbb N}
\newcommand{\Z}{\mathbb Z}
\newcommand{\E}{\mathbb E}
\newcommand{\eins}{\mathbbm 1}
\DeclareMathOperator{\Var}{Var}
\newcommand{\BB}{\mathcal{B}}
\newcommand{\QQ}{\mathcal{Q}}

\DeclareMathOperator{\disc}{disc}
\def\calD{\mathcal{D}}

\allowdisplaybreaks[4]

\usepackage{hyperref}

\begin{document}
\sloppypar

\title{A Sharp Discrepancy Bound for Jittered Sampling}

\author{Benjamin Doerr\\ Laboratoire d'Informatique (LIX)\\ CNRS\\ \'Ecole Polytechnique\\ Institut Polytechnique de Paris\\ Palaiseau\\ France}

\maketitle

\begin{abstract}
  For $m, d \in {\mathbb N}$, a jittered (or stratified) sampling point set $P$ having $N = m^d$ points in $[0,1)^d$ is constructed by partitioning the unit cube $[0,1)^d$ into $m^d$ axis-aligned cubes of equal size and then placing one point independently and uniformly at random in each cube. We show that there are constants $c > 0$ and $C$ such that for all $d$ and all $m \ge d$ the expected non-normalized star discrepancy of a jittered sampling point set satisfies \[c \,dm^{\frac{d-1}{2}} \sqrt{1 + \log(\tfrac md)} \le {\mathbb E} D^*(P) \le C\, dm^{\frac{d-1}{2}} \sqrt{1 + \log(\tfrac md)}.\]
  This discrepancy is thus smaller by a factor of $\Theta\big(\big(\frac{1+\log(m/d)}{m/d}\big)^{1/2}\big)$ than the one of a uniformly distributed random point set (Monte Carlo point set) of cardinality~$m^d$. This result improves both the upper and the lower bound for the discrepancy of jittered sampling given by Pausinger and Steinerberger (Journal of Complexity~(2016)). It also removes the asymptotic requirement that $m$ is sufficiently large compared to $d$.   
\end{abstract}

\section{Introduction}

\subsection{Star Discrepancy}

The \emph{star discrepancy} $D^*(P)$ of a set $P$ of $N$ points in the $d$-dimensional unit cube $[0,1)^d$ is a measure for the uniformity of the distribution of these points. It is defined by 
\begin{equation}\label{eq:disc}
D^*(P) := \sup_{B \in \BB} \big| |P \cap B| - N \lambda(B) \big|,
\end{equation}
where $\lambda(\cdot)$ denotes the Lebesgue measure and $\BB$ is the set of all axis-parallel rectangles $[0,x) := \prod_{i=1}^d [0,x_i)$, $x = (x_1,\dots,x_d) \in [0,1)^d$. We call $[0,x)$ a rectangle anchored in the origin or a \emph{box}. The star discrepancy thus is a worst-case measure for how well $P$ satisfies the target of having the fair number $N\lambda(B)$ of points in each box $B$. 

Evenly distributed points sets and the star discrepancy have found applications in various areas such as machine learning~\cite{AvronSYM16}, heuristic search~\cite{KimuraM05,OwenT05,TeytaudG07}, statistics~\cite{FangW93}, and computer graphics~\cite{Owen03}. Most prominent is its role in numerical integration, where the \emph{Koksma-Hlawka inequality}~\cite{Koksma42,Hlawka61} bounds the integration error in terms of the star discrepancy: For all functions $f : [0,1]^d \to \R$ having variation in the sense of Hardy and Krause bounded by $1$, we have \[\bigg| \int_{[0,1]^d} f(x) dx - \frac 1 {|P|} \sum_{p \in P} f(p) \bigg| \le \frac 1 {|P|} D^*(P).\]

We remark that in this work we use the non-normalized version of the star discrepancy as defined above, which is more common in those areas of discrepancy theory that aim at a unified view on discrepancies and exploit connections between different discrepancy notions, e.g., between geometric discrepancies like the star discrepancy and combinatorial discrepancies like hypergraph discrepancies. See the books of Matou\v sek~\cite{Matousek99} and Chazelle~\cite{Chazelle00} for an introduction to this field. In contexts closer related to numerical integration, motivated by results like the Koksma-Hlawka inequality, a normalized version of the star discrepancy is more common. The normalized star discrepancy is exacty the notion of~\eqref{eq:disc} multiplied by a factor of $\frac 1 {|P|}$. There is little risk of confusion since the normalized version is always at most one and usually far less than one, whereas our notion is at least $1/2$ and usually much larger than $1$.

\subsection{Estimates for the Star Discrepancy} 

The interest in discrepancies from various research communities has led to a huge body of research (see, e.g., \cite{Chazelle00,DickP10,DrmotaT97,Matousek99,Niederreiter92}), which we cannot fully review here. The classic view on geometric discrepancies is to treat the dimension $d$ as a constant and investigate the asymptotic behavior for growing numbers $N$ of points. In this view, a large number of constructions of point sets has been exhibited that have a discrepancy of $D^*(P) = O(\log(N)^{d-1})$. It is also known that such a polylogarithmic discrepancy cannot be avoided, though the optimal exponent is not known and finding it is a famous open problem, see, e.g.,~\cite{BilykL13}. As a side remark, we note that axis-parallel regular grids have a discrepancy of order $\Theta(N^{(d-1)/d})$ and uniformly distributed random point sets (Monte Carlo point sets) have a discrepancy of order $\Theta(\sqrt N)$, both in expectation and with high probability. Thus both are not competitive when treating $d$ as a constant and looking for asymptotic discrepancy guarantees in terms on~$N$.

From the viewpoint of numerical integration in high dimension, a behavior exponential in $d$ like $(\log N)^{d-1}$ is not very desirable, since such bounds become interesting often only for numbers $N$ of points that are far beyond any practical meaning. For this reason, Heinrich, Novak, Wasilkowski, and Wo\'{z}niakowski~\cite{HeinrichNWW01} started the quest for discrepancy bounds that both make the dependence on $d$ fully explicit and that give reasonable discrepancy guarantees also when $N$ is only of moderate size compared to $d$. Interestingly, this brought random constructions back on stage, and in fact, they are at the moment the best constructions in the regime where $N$ is not very large (say exponential) compared to $d$. 

In~\cite{HeinrichNWW01}, Heinrich et al.\ prove that uniformly distributed random point sets have an expected discrepancy of order $\sqrt{dN}$, that is, there is a constant $C$ such that for all $d$ and $N$ a set $P$ of $N$ points chosen independently and uniformly at random in $[0,1)^d$ satisfies $\E D^*(P) \le C \sqrt{dN}$. This is asymptotically tight~\cite{Doerr14} in the sense that  there is a constant $c>0$ such that for all $d$ and all $N \ge d$ the corresponding uniformly distributed random point set satisfies $\E D^*(P) \ge c \sqrt{dN}$. 

Determining the leading constant $C$ remains a major open problem. For the upper bound, the original proof of Heinrich et al.\ does not easily reveal information on~$C$.  Aistleitner~\cite{Aistleitner11} gave an alternative, more direct proof that also shows that with positive probability, $D^*(P) \le 10 \sqrt{dN}$. The currently strongest estimate, lowering the $10$ to $2.525$, is due to Gnewuch and Hebbinghaus~\cite{GnewuchH21}. For the lower bound, the elementary proof of~\cite{Doerr14} clearly can be made more precise and then give a reasonable constant, but this has not been done so far. 

\subsection{Jittered Sampling}

Given the success of random point sets, it is natural to think of constructions that employ randomness in a more clever way than just by taking all decisions independently and uniformly at random. The two most prominent dependent randomized constructions are \emph{Latin hypercube samplings}~\cite{McKayBC79} and  \emph{jittered sampling} (also called \emph{stratified sampling})~\cite{Bellhouse81,CookPC84}. While the discrepancy of Latin hypercube samples was analyzed only recently~\cite{DoerrDG18,GnewuchH21} in the paradigm of not treating $d$ as a constant, the first such analysis for jittered sampling was conducted already in 2004. 

Assume that we can write $N = m^d$ for some integer~$m$. To obtain a random $N$-point set $P$ via jittered sampling, we partition the unit cube $[0,1)^d$ into $m^d$ axis-parallel cubes of identical size and, independently, place a uniformly distributed random point  in each cube. 

We formulate all results in the following in terms of $m$ and $d$, and recall that the number $N$ of points is $N = m^d$. To ease comparing the different results, we write all bounds with an explicit term $\sqrt{dm^d}$ (or $\sqrt{m^d}$ when $d$ is treated as a constant), which is the order of magnitude of the discrepancy of a uniform random point set. For constant $d$, Beck~\cite{Beck87} showed that the expected discrepancy of such a point set $P$ satisfies
\[
\E D^*(P) = O\left(\sqrt{m^d} \sqrt{\frac{\log m}{m}}\right).
\]

In the technical report~\cite{DoerrGS04} (also described in~\cite[(28)]{Gnewuch12}), an upper bound of 
\[
O\left(\sqrt{dm^d} \sqrt{\frac{\log N}{m/d}}\,\right) = O\left(\sqrt{dm^d} \sqrt{d\,\frac{\log m}{m/d}}\,\right)
\] 
was shown. This bound extends the result of Beck~\cite{Beck87} to non-constant $d$, but the dependence of the discrepancy on $d$ is weak due to use of the (as we know from~\cite{Gnewuch08}) non-optimal $\delta$-covers from~\cite{DoerrGS05} and the absence of Aistleitner's~\cite{Aistleitner11} dyadic chaining method.

A significant improvement of the dependence on $d$ and the first lower bound was presented by Pausinger and Steinerberger~\cite{PausingerS16}, who proved that for all $d \in \N$ and $m$ sufficiently large relative to $d$, 
\begin{equation}\label{eq:proven}
\frac 1 {10} \sqrt{d m^d} \sqrt{\frac 1 {m/d}} \le \E D^*(P) \le \sqrt{d m^d} \sqrt{\frac{\log m}{m/d}}.
\end{equation}
It is not stated in the paper for which values of $m$ and $d$ this bound is valid, that is, what ``$m$ sufficiently large compared to $d$'' precisely means, and the authors state it as an open problem to overcome this asymptotic requirement.  

Concerning the lower bound, Pausinger and Steinerberger conjecture that it is not tight and speculate that a lower bound of 
\begin{equation}\label{eq:conj}
\Omega\bigg(\sqrt{d m^d} \, \frac{1 + \sqrt{\log(m)/d}}{\sqrt{m/d}}\bigg)
\end{equation}
``might actually be very close to the truth.'' Note that this bound is asymptotically stronger than the previous lower bound only for $m$ superexponential in $d$, that is, $m = 2^{\omega(d)}$.

\subsection{Our Result: A Tight Discrepancy Estimate for Jittered Sampling}

We show that both the conjectured lower bound~\eqref{eq:conj}, and thus also the proven lower bound in~\eqref{eq:proven}, and the proven upper bound in~\eqref{eq:proven} are not tight, but that instead the true order of magnitude is \[\Theta\left(\sqrt{d m^d} \sqrt{\frac{1+\log(m/d)}{m/d}}\,\right)\] for all $m$ and $d$ such that $m \ge d$. 

Our result shows that it is the ratio of $m$ and $d$ that describes by how much better jittered sampling is compared to uniform random sampling. For $m = \Theta(d)$, the expected star discrepancies are asymptotically the same, namely $\Theta(\sqrt{dm^d}\,)$. When $m = \omega(d)$,  jittered sampling is superior, leading to discrepancies smaller by a factor of $\Theta\Big(\sqrt{\frac{\log(m/d)}{m/d}}\, \Big)$. Note that the upper bound of~\cite{PausingerS16} shows an advantage of jittered sampling only for $m = \omega(d \log d)$ and ``$m$ large enough compared to $d$''. 

In this article, we did not aim at making the leading constant or any lower order terms precise, though in principle we do not see any obstacles for obtaining reasonable absolute bounds for both the lower bound (mainly relying on an estimate on the maximum of independent binomial random variables) and the upper bound (mainly relying on Aistleitner's dyadic chaining method). 

Our result suggests (but we do not prove this) that for $m < d$, jittered sampling does not lead to discrepancies of asymptotic order smaller than those of independent random point sets. Since there is good reason to believe that jittered sampling nevertheless has some (small, but gratuitous) advantage, we discuss brief{}ly in Section~\ref{sec:small} how to use jittered sampling also for smaller number of points (including less than $2^d$) and what can be said about the discrepancy of such point sets.



\section{Notation and Preliminaries}

Throughout this work, we use the following notation: Given an $N$-point set $P \subseteq [0,1)^d$, we denote for each Lebesgue-measurable set $A$  by $\lambda(A)$ its Lebesgue measure and by  \[\disc(A) := \disc_P(A) := |P \cap A| - N \lambda(A)\] the \emph{signed non-normalized discrepancy} of the set $A$. 

The (non-normalized) \emph{star discrepancy} of the point set $P$ is $D^*(P) := \sup_B |\disc(B)|$, where $B$ runs over all \emph{axis-aligned rectangles} (boxes) with lower left corner in the origin, that is, all sets $[0,x) := \prod_{i=1}^d [0,x_i)$, $x \in [0,1)^d$.

For a positive integer $n$, we use $[n] := \{1, \dots, n\}$ as shorthand for the set of the first $n$ positive integers. We also write $[a..b] := \{z \in \Z \mid a \le z \le b\}$. 

For given $m$ and (usually suppressed) $d$, we call $G_m := \{0, \frac 1m, \dots, \frac{m-1}{m}\}^d$ the \emph{$m$-grid} (in dimension $d$). For $x \in G_m$, we call $C_x := [x,x+\frac 1m \mathbf{1}_m) = \prod_{i=1}^d [x_i,x_i+\frac 1m)$ an \emph{$m$-cube} (or \emph{cube}) in dimension $d$. A \emph{jittered sampling point set} of $N = m^d$ points is obtained by taking independently and uniformly at random one point from each $m$-cube. 

We now prove an elementary lower bound for tail probabilities of the binomial distribution in the special case $p=1/2$. Such estimates can be proven via normal approximations and the Berry-Esseen theorem, see, e.g.,~\cite[XVI.5]{Feller71}. To avoid such deep methods in this otherwise elementary combinatorial paper, we now show the following estimate. 

\begin{lemma}\label{lem:maxbin}
  Let $k \in \N$. Let $X_1, \dots, X_k$ be independent random variables each having a binomial distribution with parameters $n$ and $\frac 12$. We denote the maximum of these by $X_{\max} := \max\{X_i \mid i \in [k]\}$. If $c \in \R$ is such that 
	\[
	\alpha(c) := \sqrt{\frac{n (\ln k - \frac 12 \ln\ln k - c)}{2(1+\sqrt{2\ln(k)/n}\,)}}
	\] 
	is at least $\sqrt n$ and small enough to satisfy $\alpha(c) + \frac n {\alpha(c)} \le \frac n2$, then 
	\[
	\Pr[X_{\max} \ge \tfrac n2 + \alpha(c)] \ge 1 - \exp\bigg(-\frac 1 {1.5 e^{169/6} \sqrt {\pi}}  e^{c}\bigg).
	\]
Consequently, for $e^6 \le k \le e^{n/2}$, we have 
\[
\E[\max\{0,X_{\max}- \tfrac n2\}] \ge \sqrt{\frac{n (\ln k - \ln\ln k)}{2(1+\sqrt{2\ln(k)/n}\,)}} \bigg(1 - \exp\bigg(-\frac {\sqrt{\ln k}} {1.5 e^{169/6} \sqrt {\pi}}\bigg)\bigg).
\]
\end{lemma}

\begin{proof}
  Let $X$ be a random variable having a binomial distribution with parameters $n$ and $\frac 12$. Let $\alpha \ge 0$ be such that $\frac n2+\alpha \in \N$. We first give a lower bound for the probability that $X$ exceeds its expectation by exactly $\alpha$. Using Stirling's approximation in the version 
	\[
	\sqrt {2\pi} n^{n+\frac 12} e^{-n} e^{\frac 1 {12n+1}} < n! < \sqrt {2\pi} n^{n+\frac 12} e^{-n} e^{\frac 1 {12n}}
	\] 
	for all $n \ge 1$ due to Robbins~\cite{Robbins55} and the elementary estimate $1+x \le e^x$ valid for all $x \in \R$, we compute for $\alpha < \frac n2$ that
\begin{align*}
  \Pr[&X = \tfrac n2 + \alpha] = 2^{-n} \binom{n}{ \frac n2 + \alpha}\\
  & = 2^{-n} \frac{n!}{(\frac n2+\alpha)! \, (\frac n2 - \alpha)!}\\
  & \ge \frac 1 {\sqrt {2\pi}\sqrt n (1 + \frac{2\alpha}{n})^{n/2 + \alpha + 1/2} (1 - \frac{2\alpha}{n})^{n/2 - \alpha + 1/2}} \frac{\exp(\frac 1{12n+1})}{\exp(\frac 1{12(\frac n2 + \alpha)}) \exp(\frac 1{12(\frac n2 - \alpha)}) }\\
  & \ge \frac 1 {\sqrt {2\pi}} \frac 1 {\sqrt n (1 + \frac{2\alpha}{n})^{2\alpha} (1 - (\frac{2\alpha}{n})^2)^{n/2 - \alpha + 1/2}} \frac{1}{\exp(\frac 1{12}) \exp(\frac 1{12}) }\\
  & \ge \frac 1 {e^{1/6} \sqrt {2\pi}} \frac 1 {\sqrt n \exp((2\alpha)^2 / n) \exp(-(2\alpha/n)^2 (n/2 - \alpha + 1/2))}\\
  & \ge \frac 1 {e^{1/6} \sqrt {2\pi}} \frac 1 {\sqrt {n}} e^{-2\alpha^2 / n - 4 \alpha^3/n^2}.
\end{align*}
Note that the lower bound in the last line is also valid for $\alpha = \frac n2$, simply because it is less than $2^{-n}$.
  
  Let $\alpha \ge \sqrt n$ with $\alpha + \frac n \alpha \le \frac n2$. We now estimate the probability that $X$ exceeds its expectation by at least $\alpha$ via the probability that $X \in [\E X + \alpha, \E X + \alpha + \frac n\alpha]$. Using that $\binom{n}{\frac n2+\alpha}$ and thus $\Pr[X = \frac n2+\alpha]$ is decreasing in $\alpha$ for integral $\alpha \ge 0$, we compute
\begin{align}
  \Pr[X \ge \tfrac n2 + \alpha] & \ge \Pr[\tfrac n2 + \alpha \le X \le \tfrac n2 + \alpha + \tfrac n\alpha]\nonumber\\
  & \ge \lfloor \tfrac n\alpha \rfloor  \Pr[X = \lfloor \tfrac n2 + \alpha + \tfrac n\alpha \rfloor]\nonumber\\
  & \ge \frac 1 {e^{1/6} \sqrt {2\pi}} \frac{\lfloor \tfrac n\alpha \rfloor}{\sqrt n} e^{-2(\alpha+n/\alpha)^2 / n- 4 (\alpha+n/\alpha)^3/n^2}\label{eq:bino}\\ 
  & \ge \frac 1 {1.5 e^{169/6} \sqrt {2\pi}} \frac {\sqrt {n}}{\alpha} e^{-2\alpha^2 / n - 4\alpha^3/n^2}, \nonumber
\end{align}
where the last inequality uses $\alpha \ge \sqrt n$, $\alpha \le \frac n2$, and $\alpha \ge 1$, which give the estimates $2(\alpha+n/\alpha)^2 / n = 2 \frac{\alpha^2}{n} + 4 + 2 \frac n {\alpha^2} \le 2 \frac{\alpha^2}{n} + 6$ and $4 (\alpha+n/\alpha)^3/n^2 = 4 \frac{\alpha^3}{n^2} + 12 \frac \alpha n + 12 \frac 1\alpha + 4 \frac n{\alpha^3} \le 4 \frac{\alpha^3}{n^2} + 6 + 12 + \frac 4\alpha \le 4 \frac{\alpha^3}{n^2} + 22$.

Consequently, the probability that all of $X_1, \dots, X_k$ are below $\frac n2 + \alpha$ is
\begin{align*}
  \Pr[X_{\max} < \tfrac n2+\alpha] & = \Pr[\forall i \in [k] : X_i < \tfrac n2 + \alpha]\\
  & \le \bigg(1 - \frac 1 {1.5 e^{169/6} \sqrt {2\pi}} \frac {\sqrt {n}}{\alpha} e^{-2\alpha^2 / n - 4\alpha^3/n^2}\bigg)^k\\
  & \le \exp\bigg(-\frac k {1.5 e^{169/6} \sqrt {2\pi}} \frac {\sqrt {n}}{\alpha} e^{-2\alpha^2 / n - 4\alpha^3/n^2}\bigg).
\end{align*}

Let $c \in \R$ be such that 
\begin{align*}
\alpha = \alpha(c) = \sqrt{\frac{\tfrac 12 n (\ln k - \tfrac 12 \ln\ln k - c)}{1+\sqrt{2\ln(k)/n}}}
\end{align*}
satisfies $\alpha \ge \sqrt n$ and $\alpha + \frac n \alpha \le \frac n2$. With $\alpha \le \sqrt{\frac 12 n \ln k }$, we continue the previous estimate and compute
\begin{align*}
  \exp&\bigg(-\frac k {1.5 e^{169/6} \sqrt {2\pi}} \frac {\sqrt {n}}{\alpha} e^{-(2\alpha^2 / n)(1+2\alpha/n)}\bigg)\\
  & \le \exp\bigg(-\frac k {1.5 e^{169/6} \sqrt {2\pi}} \frac {\sqrt {n}}{\sqrt{\frac 12 n \ln k}} \exp\bigg(-\frac{(\ln k - \tfrac 12 \ln\ln k - c)(1+2\alpha/n)}{1+\sqrt{2\ln(k)/n}}\bigg)\bigg)\\
  & \le \exp\bigg(-\frac k {1.5 e^{169/6} \sqrt {\pi}} \frac {1}{\sqrt{\ln k}} \exp(-(\ln k - \tfrac 12 \ln\ln k - c))\bigg)\\
  & = \exp\bigg(-\frac 1 {1.5 e^{169/6} \sqrt {\pi}}  e^{c}\bigg) .
\end{align*}

To prove the second claim, where $e^6 \le k \le e^{n/2}$, let $c_k = \tfrac 12 \ln\ln k$ and 
\[
\alpha = \alpha_k = \alpha(c_k) = \sqrt{\frac{\frac 12 n (\ln k - \ln\ln k)}{1 + \sqrt{2 \ln(k)/n}}}.
\]
We note that for $k \in [e^6,e^{n/2}]$, we have $\alpha_k \ge \sqrt{\frac{\frac 12 n (\ln e^6 - \ln\ln e^6)}{1 + \sqrt{2 \ln(e^{n/2})/n}}} \ge \sqrt n$; here we used that $k \mapsto \ln k - \ln\ln k$ is increasing for $k \ge e$. For $\alpha \ge \sqrt n$, the expression $\alpha + \frac n\alpha$ is increasing in $\alpha$. Hence noting that $\alpha_k \le \sqrt{\frac{\frac 12 n \ln k}{1 + \sqrt{2 \ln(k)/n}}} = \frac{n}{2} \sqrt{\frac{2\ln(k)/n}{1+\sqrt{2 \ln(k)/n}}} \le \frac{n}{2} \sqrt{\frac{2\ln(e^{n/2})/n}{1+\sqrt{2 \ln(e^{n/2})/n}}} \le \frac{n}{\sqrt 8}$ when $k \le e^{n/2}$ -- here we used the fact that $x \mapsto \frac{x}{1+\sqrt x}$ is increasing in $\R_{\ge 0}$ --, we estimate $\alpha_k + \frac n{\alpha_k} \le \frac{n}{\sqrt 8} + \frac n {n / \sqrt 8} \le \frac n{\sqrt 8} + \sqrt 8$, which is at most $\frac n2$ when $n \ge 20$. Hence for $n\ge 20$ and any $k \in [e^6,e^{n/2}]$, we can use the first claim of this lemma and compute 
\begin{align*}
\E[\max\{0,X_{\max} - \tfrac n2\}] &\ge \alpha_k \cdot \Pr[X_{\max} \ge \tfrac n2 + \alpha_k] \\
&\ge \sqrt{\frac{n (\ln k - \ln\ln k)}{2(1+\sqrt{2\ln(k)/n}\,)}} \bigg(1 - \exp\bigg(-\frac {\sqrt{\ln k}} {1.5 e^{169/6} \sqrt {\pi}}\bigg)\bigg).
\end{align*}
For $n < 20$, the second claim is trivially fulfilled as the following two estimates show (where the latter again uses $e^x \ge 1 + x$, valid for all $x \in \R$).
\begin{align*}
&E[\max\{0,X_{\max} - \tfrac n2\}] \ge E[\max\{0,X_1 - \tfrac n2\}] \ge \tfrac n2 \Pr[X_1 = n] = 2^{-n} \tfrac n2 \ge 2^{-19} \tfrac n2.\\
&\sqrt{\frac{n (\ln k - \ln\ln k)}{2(1+\sqrt{2\ln(k)/n}\,)}} \bigg(1 - \exp\bigg(-\frac {\sqrt{\ln k}} {1.5 e^{169/6} \sqrt {\pi}}\bigg)\bigg) \\
& \quad \le \sqrt{\frac{n \ln k}{2}} \bigg(1 - 1 - \bigg(\frac {\sqrt{\ln k}} {1.5 e^{169/6} \sqrt {\pi}}\bigg)\bigg)  \\
& \quad \le \sqrt{\tfrac 12 n \ln e^{n/2}} \frac{\sqrt{\ln e^{n/2}}}{1.5 e^{169/6} \sqrt {\pi}} < \tfrac n2 \sqrt 8 e^{-28} < 2^{-19} \tfrac n2.
\end{align*}
\end{proof}

\section{Proof of the Lower Bound}

In this section, we prove that for all $m \ge d$ the discrepancy of a jittered sampling point set having $m^d$ point in $[0,1)^d$ is at least of order $\sqrt{dm^d} \sqrt{\frac{1+\log(m/d)}{m/d}}$. 
\begin{theorem}\label{thm:lower}
  There is a constant $C>0$ such that for all $m, d \in \N_{\ge 2}$ with $m \ge d$, the expected discrepancy of a jittered sampling point set $P \subset [0,1)^d$, $|P|=m^d$, is at least \[\E D^*(P) \ge C \sqrt{dm^d} \sqrt{\frac{1+\log(m/d)}{m/d}}.\]
\end{theorem}

To ease the presentation, we treat the ``small'' case that $m$ is at most a constant factor larger than $d$ separately in Lemma~\ref{lem:smallm}. The more interesting case, naturally, is that $m$ is of larger order than $d$. For this, we prove the following result, which we state in a non-asymptotic fashion, noting again that in this work we did not optimize the leading constant or the lower order terms. 

\begin{lemma}\label{lem:lowermain}
  Let $m, d \in \N_{\ge 2}$ with $\lfloor \frac md \rfloor \ge e^6$, and $N = m^d$. Let $P$ be a random $N$-point set in $[0,1)^d$ obtained from jittered sampling. Then 
  \begin{align*}
  \E D^*(P) &\ge (2e)^{-\frac 12} d m^{\frac{d-1}2} \sqrt{\ln(\lfloor\tfrac md \rfloor) - \ln\ln(\lfloor \tfrac md \rfloor)} \\ 
  &\quad  \left(1+\sqrt{\frac{2\ln(\lfloor \tfrac md \rfloor)}{(m - \lfloor \tfrac md \rfloor)^{d-1}}}\,\right)^{-\frac 12} \left(1 - \exp\left(-\frac {\sqrt{\ln \lfloor \tfrac md \rfloor}} {1.5 e^{169/6} \sqrt {\pi}}\,\right)\right).
  \end{align*}
  In particular, there is a $C > 0$ such that for all $m, d \in \N_{\ge 2}$ with $\lfloor \frac md \rfloor \ge e^6$ a jittered sampling point set $P$ with $N = m^d$ points satisfies $\E D^*(P) \ge C \sqrt{dm^d} \sqrt{\frac{1+\log(m/d)}{m/d}}$.
\end{lemma}

While we did not optimize for the leading constant, our result can be written as $\E D^*(P) \ge (1-f(\frac md)) (2e)^{-\frac 12} \sqrt{dm^d} \sqrt{\frac{\ln(m/d)}{m/d}}$, where $f$ is a function tending to zero when the argument tends to infinity. Consequently, our leading constant of $(2e)^{-\frac12} \ge 0.4288$ is not too bad. We note that Pausinger and Steinerberger~\cite{PausingerS16} state a constant of $1/10$ in their theorem. An inspection of their proof shows that they actually prove their result with a leading constant of $0.5 \sqrt{\pi/2} \ln(2) = 0.4343...$.

In the proof of Lemma~\ref{lem:lowermain}, we use the elementary observation that any measurable set has expected signed discrepancy zero when all cubes intersecting it contain exactly one random point distributed uniformly in the cube. 

\begin{lemma}\label{lem:discnull}
  Let $A \subseteq [0,1)^d$ be a measurable set. Let $P$ be a set of $N = m^d$ random points such that each $m$-cube having non-empty intersection with $A$ contains exactly one point of $P$ and this point is uniformly distributed in this cube. Then $\E \disc(A) = 0$.
\end{lemma}

\begin{proof}
  Let $\QQ$ be the set of all cubes and $\QQ_A$ the set of cubes having non-empty intersection with $A$. Then
  \begin{align*}
    \disc(A) &= \sum_{Q \in \QQ} \disc(A \cap Q) = \sum_{Q \in \QQ_A} \disc(A \cap Q).
  \end{align*}
  Let $Q \in \QQ_A$. With probability $\lambda(A \cap Q) / \lambda(Q)$, the random point in $Q$ lies also in $A$ and we have $\disc(A \cap Q) = 1 - N \lambda(A \cap Q)$. Otherwise, $P \cap (A \cap Q)$ is empty, giving $\disc(A \cap Q)  = - N \lambda(A \cap Q)$. Consequently, 
  \begin{align*}
  \E \disc(A \cap Q) &= \frac{\lambda(A \cap Q)}{\lambda(Q)} (1 - N \lambda(A \cap Q)) - \bigg(1 - \frac{\lambda(A \cap Q)}{\lambda(Q)}\bigg) N \lambda(A \cap Q) \\
  &= \frac{\lambda(A \cap Q)}{\lambda(Q)} - N \lambda(A \cap Q) = 0,
  \end{align*}
  where the last equality follows from $\lambda(Q) = 1/N$. By linearity of expectation, $\E \disc(A) = \sum_{Q \in \QQ_A} \E \disc(A \cap Q) = 0$.
\end{proof}
 
We shall use the above observation to combine certain rectangles with known signed discrepancy into an anchored box with expected discrepancy signed equal to the sum of the discrepancies of these rectangles. 

\begin{lemma}\label{lem:construct}
  Let $r_1, \dots, r_d \in [0,1)$ be integer multiples of $1/m$. Let $B_0 := [0,r)$, where $r = (r_1,\dots,r_d)$. Let $P$ be a random set of $m^d$ points in $[0,1)^d$ obtained from jittered sampling. For each $i \in [d]$, let $S_i = S_i(P)$ be a random variable taking values in $[r_i,1)$ and let $R_i = \prod_{j=1}^{i-1} [0,r_j) \times [r_i,S_i) \times \prod_{j=i+1}^d [0,r_j)$. Assume that each $R_i$, i.e., the number~$S_i$, is independent of the position of all points outside $\bar R_i :=  \prod_{j=1}^{i-1} [0,r_j) \times [r_i,1) \times \prod_{j=i+1}^d [0,r_j)$. Let $B = \prod_{i=1}^d [0,S_i)$ be the smallest anchored box containing the $R_i$. Then \[\E D^*(P) \ge \E \disc(B) = \sum_{i=1}^d \E \disc(R_i).\]
\end{lemma}

Note that the lemma in particular covers the case that the $S_i$ are chosen as to maximize the signed discrepancy of the $R_i$. Consequently, with this lemma we can construct a box with large (expected) discrepancy by finding $R_i$ with large signed discrepancy. We speculate that this \emph{construction principle} can be useful in other lower bound proofs for jittered sampling as well.

\begin{proof}[Proof of Lemma~\ref{lem:construct}]
  Let us first condition on a fixed outcome of the $S_i$, that is, let $s_i \in [r_i,1)$ for all $i \in [1..d]$ and we condition on $S_i = s_i$ for all $i \in [1..d]$. In this conditional probability space, by construction, all points in cubes not contained in $\bar R := \bigcup_{i \in [d]} \bar R_i$ are uniformly distributed in their cube. Also, no cube intersects both $A := B \setminus \bigcup_{i \in [d]} R_i$ and $\bar R$. Hence all points in cubes with non-empty intersection with $A$ are distributed uniformly in their cube. By Lemma~\ref{lem:discnull} we have $\E[\disc(A) \mid S_1 = s_1, \dots, S_d = s_d] = 0$. We thus have
\begin{align*}
  \E[&\disc(B) \mid S_1 = s_1, \dots, S_d = s_d] \\ 
	&= \E\left[\disc(A) + \sum_{i=1}^d \disc(R_i) \,\middle|\, S_1 = s_1, \dots, S_d = s_d\right] \\
	&= \sum_{i=1}^d \E[\disc(R_i) \mid S_1 = s_1, \dots, S_d = s_d].
\end{align*}
Hence the law of total expectation gives
\begin{align*}
  \E\disc(B) & = \E[\E[\disc(B) \mid S_1 = s_1, \dots, S_d = s_d]] \\
	&= \E\left[\sum_{i=1}^d \E[\disc(R_i) \mid S_1 = s_1, \dots, S_d = s_d]\right]\\
	&= \sum_{i=1}^d \E[\E[\disc(R_i) \mid S_1 = s_1, \dots, S_d = s_d]]\\
	&= \sum_{i=1}^d \E \disc(R_i).
\end{align*}
	We thus have $\E D^*(P) \ge \E|\disc(B)| \ge \E\disc(B) = \sum_{i=1}^d \E\disc(R_i)$.
\end{proof}

Before giving the precise proof of Lemma~\ref{lem:lowermain}, let us give a brief outline of the main ideas and compare them to the proof of Pausinger and Steinerberger~\cite{PausingerS16}. The main argument of~\cite{PausingerS16} is the following. Let $x_1, \dots, x_d \in [\frac{m-1}{m},1)$ and $x = (x_1,\dots,x_d)$. If $m$ is sufficiently large compared to $d$, then the discrepancy of the box $B = [0,x)$ is very close to the discrepancy of the union $R := R_1 \cup \dots \cup R_d$ of the slices $R_i := R_i(x_i) := [0,\frac{m-1}m)^{i-1} \times [\frac{m-1}m,x_i) \times [0,\frac{m-1}m)^{d-i}$. Note that here the above lemma would have directly shown that the expected star discrepancy of $P$ is at least the discrepancy of $R$. By construction, the discrepancy of $R_i(x_i)$ has the same distribution as the discrepancy of the interval $[0,m(x_i-\frac{m-1}m))$ in a one-dimensional uniformly distributed random set of $(m-1)^{d-1}$ points in $[0,1)$. This one-dimensional discrepancy problem can be analyzed quite well, in particular, the expected maximum discrepancy (over all choices of $x_i$) can be determined. Consequently, the expected maximum discrepancy of a suitable choice of $B$ is at least $d/2$ times this number (the factor of $1/2$ stems from the fact that we need the discrepancies of the  $R_i$ to have the same sign).

From a broader perspective, the main idea of~\cite{PausingerS16} is to regard all boxes $B = [0,x)$ with $x$ lying in the upper right sub-cube $C^+ := [\frac{m-1}m,1)^d$ -- and only these -- and to then exploit that for these the discrepancy is well described by one-dimensional discrepancies, which can be analyzed very precisely. While the reduction to the one-dimensional discrepancy problem allow a very precise analysis of the maximum discrepancy of a box $[0,x)$, $x \in C^+$, this approach carries the risk that the restricted choice of boxes underestimates the star discrepancy significantly. 

For this reason, we follow a different road. We do not restrict ourselves to boxes $[0,x)$, $x \in C^+$, but overcome the increased complexity of the larger range for $x$ by restricting ourselves to a suitable \emph{discrete} set of choices for $x$. Taking $r = 1 - \frac 1m \lfloor \frac md \rfloor  \approx 1- \frac 1d$, our $x$ will be such that all $x_i$ are in $[r,1)$ and are integral multiples of $\frac 1 {2m}$. Clearly, with this relatively small discrete set of boxes, our bounds will necessarily be off the truth by constant factors. However, the more diverse set of boxes together with the right (not very difficult) combinatorial way of selecting a large-discrepancy box among them will enable us to prove the stronger (and in fact asymptotically tight) lower bound. 

The main combinatorial observation is that for all $j =0, \dots, \lfloor \frac md \rfloor-1$, the rectangle $R'_1 = [r+ \frac jm,r+ \frac jm+ \frac{1}{2m}) \times [0,r)^{d-1}$ has the same discrepancy distribution, which is the deviation of a binomial random variable with parameters $N' \ge \frac N{em}$ and $p=\frac 12$ from its expectation. Consequently, by elementary properties of the maximum of $\lfloor \frac md \rfloor$ independent binomial random variables (Lemma~\ref{lem:maxbin}), with good probability there is a choice for $j$ such that the signed discrepancy of $R'_1$ satisfies $\disc(R'_1) \ge C \sqrt{N \log( \frac md)}$ for some absolute constant $C$. By construction, the discrepancy of $R'_1$ is identical to the one of $R_1 := [r,r+ \frac jm + \frac 1{2m}) \times [0,r)^{d-1}$. Repeating this argument in each dimension and taking as $B$ the smallest anchored box that contains all $R_i$, with Lemma~\ref{lem:construct} again we obtain a box with $\E \disc(B) \ge C d \sqrt{N
 \log(\frac md)}$ as desired.

\begin{proof}[Proof of Lemma~\ref{lem:lowermain}]
Let $k := \lfloor \frac md \rfloor$ and $r = \frac{m-k}{m} \ge 1-\frac 1d$. For each $i \in [d]$ let $S_i \in [r_i,1)$ be (maximal, in case of ambiguity) such that $R_i := [0,r)^{i-1} \times [r,S_i) \times [0,r)^{d-i}$ has maximum signed discrepancy $\disc(R_i) := |P \cap R_i| - N \lambda(R_i)$. Note that $\disc(R_i)$ can only be maximal when $S_i$ coincides with the $i$-th coordinate of some point of $P$, which shows that $S_i$ is a well-defined random variable. Note further that $\disc(R_i) \ge 0$ since taking $S_i = 1$ would give a rectangle containing only full cubes, and hence, having discrepancy zero.

Let $B := \prod_{i \in [d]} [0,S_i)$ be the smallest box containing all~$R_i$. Note that the value of the $S_i$ depends only on the position of the points in the cubes in $\bar R_i := [0,r)^{i-1} \times [r,1) \times [0,r)^{d-i}$. 
By Lemma~\ref{lem:construct}, we have
\begin{equation}
  \E D^*(P) \ge \sum_{i=1}^d \E\disc(R_i) = d \, \E \disc(R_1),\label{eq:dmal}
\end{equation}
where the last equality exploits the symmetry between the $R_i$. 

So it suffices to analyze $\disc(R_1)$. For $j = 0, \ldots, k-1$, let $y_j := r + \frac jm$ and $z_j := y_j + \frac 1{2m}$. Let $U_j := [r,z_j) \times [0,r)^{d-1}$ and $T_j = [y_j,z_j) \times [0,r)^{d-1}$. Since $U_j \setminus T_j$ can be written as union of cubes, we have $\disc(U_j) = \disc(T_j)$. Now each $T_j$ is composed of the ``left'' halves of exactly $N' = (m-k)^{d-1} \ge (1-\frac1d)^{d-1} m^{d-1} \ge \frac 1e m^{d-1}$ cubes. Consequently, $|P \cap T_j|$ follows a binomial distribution with parameters $N'$ and $p=\frac12$. Since no cube intersects non-trivially two different $T_j$, the discrepancies of the $T_j$ (and thus of the $U_j$) are independent. Let $X_j$, $j = 0, \ldots, k-1$, be independent random variables with binomial distribution with parameters $N'$ and $p=\frac12$. Then $\disc(R_1) \ge \max\{0, \disc(U_0), \dots, \disc(U_{k-1})\}$ and the latter is distributed as $\max\{0, X_0 - \E X_0, \dots, X_{k-1} - \E X_{k-1}\}$. 

Since $k \ge e^6$ by assumption and further $k \le m-k \le N' \le e^{N'/2}$, by the second part of Lemma~\ref{lem:maxbin} we have 
\begin{align*}
  \E \disc(R_1) &\ge \sqrt{\frac{N' (\ln k - \ln\ln k)}{2(1+\sqrt{2\ln(k)/N'}\,)}} \, \bigg(1 - \exp\bigg(-\frac {\sqrt{\ln k}} {1.5 e^{169/6} \sqrt {\pi}}\bigg)\bigg)\\
  &\ge (2e)^{-\frac 12} m^{\frac{d-1}{2}} \sqrt{\ln(\lfloor\tfrac md \rfloor) - \ln\ln(\lfloor \tfrac md \rfloor)} \, \left(1+\sqrt{\frac{2\ln(\lfloor \tfrac md \rfloor)}{(m - \lfloor \tfrac md \rfloor)^{d-1}}}\,\right)^{-\frac 12} \\ & \quad\bigg(1 - \exp\bigg(-\frac {\sqrt{\ln(\lfloor \tfrac md \rfloor)}} {1.5 e^{169/6} \sqrt {\pi}}\bigg)\bigg).
\end{align*}
Together with~\eqref{eq:dmal}, this proves the main claim. 
The last claim follows from the following three estimates.
\begin{itemize}
\item We note that $f: x \mapsto \sqrt x - \ln x$ has its global minimum (in $\R_{>0}$) at $x = 4$. Since $f(4) > 0$, we have $\sqrt x > \ln(x)$ for all $x \in \R_{>0}$. Consequently, $\ln\ln(\lfloor \tfrac md \rfloor) \le \frac 12 \ln(\lfloor \tfrac md \rfloor)$ and thus $\sqrt{\ln(\lfloor\tfrac md \rfloor) - \ln\ln(\lfloor \tfrac md \rfloor)} \ge \sqrt{\frac 12 \ln(\lfloor\tfrac md \rfloor)}$. Since $\lfloor\tfrac md \rfloor \ge e^6$, we can continue with $\sqrt{\frac 12 \ln(\lfloor\tfrac md \rfloor)} = \sqrt{\frac 14 \ln(\lfloor\tfrac md \rfloor^2)} \ge \frac 12 \sqrt{1+\ln(\frac md)}$.
\item From $d \ge 2$ and $\lfloor \tfrac md \rfloor \ge e^6$, we conclude
\begin{align*}
	1+\sqrt{\frac{2\ln(\lfloor \tfrac md \rfloor)}{(m - \lfloor \tfrac md \rfloor)^{d-1}}}
	& \le 1+\sqrt{\frac{2\ln(\lfloor \tfrac md \rfloor)}{(2 \lfloor \tfrac md \rfloor - \lfloor \tfrac md \rfloor)^{d-1}}}\\
	& \le 1+\sqrt{\frac{2\ln(\lfloor \tfrac md \rfloor)}{\lfloor \tfrac md \rfloor}}\\
	& \le 1+\sqrt{\frac{2\ln(e^6)}{e^6}} \le 1.18.	
\end{align*}
Here we used that $x \mapsto \frac{\ln(x)}{x}$ is decreasing for $x \ge e$.
\item Finally, $\lfloor \tfrac md \rfloor \ge e^6$ implies that
\begin{align*}
  1 - \exp\bigg(-\frac {\sqrt{\ln(\lfloor \tfrac md \rfloor)}} {1.5 e^{169/6} \sqrt {\pi}}\bigg)
	& \ge 1 - \exp\bigg(-\frac {\sqrt{\ln(e^6)}} {1.5 e^{169/6} \sqrt {\pi}}\bigg)
\end{align*}
is at least some positive constant.
\end{itemize}

\end{proof}

We now discuss the case that $m$ is of similar order of magnitude as $d$. The following result, in particular, extends the $\Omega(d m^{\frac{d-1}2})$ lower bound of Lemma~\ref{lem:lowermain} to arbitrary $m = \Theta(d)$ and shows that in this case, apart from constant factors, jittered sampling point sets have an expected discrepancy not smaller than uniformly distributed random point sets. In other words, only for $N = \omega(d)^d$ jittered sampling has a super-constant discrepancy advantage over fully random points.

\begin{lemma}\label{lem:smallm}
  Let $m, d \in \N_{\ge 2}$. Let $N = m^d$ and $P$ be a random $N$-point set obtained from jittered sampling. Then \[\E D^*(P) \ge \frac{4}{5 e^{8+1/6} \sqrt{2\pi}} \, \exp\bigg(- \frac{32}{(m-1)^{\frac{d-1}{2}}}\bigg) \, d (m-1)^\frac{d-1}{2}.\]
In particular, if $m \ge \gamma d$ for some constant $\gamma>0$, then $\E D^*(P) \ge C dm^\frac{d-1}{2}$ for some constant $C > 0$ which depends only on $\gamma$. 

If $\,\Gamma d \ge m \ge \gamma d$ for some constants $\Gamma \ge \gamma > 0$, then $\E D^*(P) \ge C' \sqrt{dN}$ for a constant $C' > 0$ which depends only on $\gamma$ and $\Gamma$.  
\end{lemma}

\begin{proof}
  We use a simplified version of the proof of Lemma~\ref{lem:lowermain}. Let $r = \frac{m-1}m$. For each $i \in [d]$, choose $x_i \in [r,1)$ such that $R_i := [0,r)^{i-1} \times [r,x_i) \times [0,r)^{d-i}$ has maximum signed discrepancy $\disc(R_i) = |P \cap R_i| - N \lambda(R_i)$. Let $x = (x_1, \ldots, x_d)$ and $B := \prod_{i \in [d]} [0,x_i)$. As in the proof of Lemma~\ref{lem:lowermain}, we have $\E D^*(P) \ge d \, \E \disc(R_1)$, so again it only remains to estimate $\E \disc(R_1)$. 
  
  Clearly, we have $\E \disc(R_1) \ge \E \max\{0, \disc(R'_1)\}$ with $R'_1 := [r,r+ \frac 1{2m}) \times [0,r)^{d-1}$. Now $\disc(R'_1)$ follows a binomial distribution with parameters $N'=(m-1)^{d-1}$ and $p = \frac 12$. If $N' \ge 16$, then for $\alpha = \sqrt{N'}$ we have $\alpha + N'/\alpha \le N'/2$ and thus equation~\eqref{eq:bino} shows that with probability at least $\frac{4}{5 e^{8+1/6} \sqrt{2\pi}} e^{-32/\sqrt{N'}}$ we have a discrepancy of $\sqrt{N'}$ or more. Consequently, $\E \disc(R_1) \ge  \E \max\{0, \disc(R'_1)\} \ge \frac{4}{5 e^{8+1/6} \sqrt{2\pi}} e^{-32/\sqrt{N'}} \sqrt{N'}$. 
	
	If $N' < 16$, then we note that $R'_1 := [r,r+\frac 1 {2N'm}]\times [0,r)^{d-1}$ with probability $1 - (1 - \frac 1 {2N'})^{N'} \ge 1 - \exp(-\frac 12) \ge 0.39$ contains at least one point. Hence $\E \disc(R_1) \ge  \E \max\{0, \disc(R'_1)\} \ge 0.39 (1 - 0.5) \ge 0.19 \ge \frac{4}{5 e^{8+1/6} \sqrt{2\pi}} \cdot 4 \ge \frac{4}{5 e^{8+1/6} \sqrt{2\pi}} e^{-32/\sqrt{N'}} \sqrt{N'}$.
	
	Hence in either case, 
	\begin{align*}
	\E D^*(P) &\ge d \, \E \disc(R_1) \ge d \frac{4}{5 e^{8+1/6} \sqrt{2\pi}} e^{-32/\sqrt{N'}} \sqrt{N'} \\
	&= \frac{4}{5 e^{8+1/6} \sqrt{2\pi}} \exp\bigg(-\frac{32}{(m-1)^{\frac{d-1}{2}}}\bigg) d (m-1)^\frac{d-1}{2}.
	\end{align*}

  If $m \ge \gamma d$ for some constant $\gamma > 0$, then using $m \ge 2$ we further estimate the above bound to $\E D^*(P) \ge \frac{4 e^{-32}}{5 e^{8+1/6} \sqrt{2\pi}}  d m^\frac{d-1}{2} (1-\frac 1m)^\frac{d-1}{2}$. Noting that $(1-\frac 1m)^m$ is increasing in $m$ and, again, that $m \ge 2$, we estimate $(1-\frac 1m)^\frac{d-1}{2} \ge (1-\frac 1m)^{m/2\gamma} \ge 2^{-1/\gamma}$. Hence there is a constant $C > 0$ depending on $\gamma$ only such that $\E D^*(P) \ge C d m^{\frac{d-1}{2}}$. When also $m \le \Gamma d$ for some constant $\Gamma$, then estimating $dm^{\frac{d-1}{2}} \ge d^{1/2} (\frac m\Gamma)^{1/2} m^{\frac{d-1}{2}} = \Gamma^{-1/2} \sqrt{dN}$ shows that we also have $\E D^*(P) \ge C' \sqrt{dN}$ for some constant $C'$ which only depends on $\gamma$ and $\Gamma$.
\end{proof}


\section{Proof of the Upper Bound}\label{sec:upper}

We now show that the lower bound proven above is tight (apart from constant factors independent of $m$ and $d$). 

\begin{theorem}\label{thm:upper}
  Let $m, d \in \N$ with $m \ge d \ge 2$. Let $N := m^d$. Let $P$ be a random set of $N$ points in $[0,1)^d$ obtained from jittered sampling. Then \[\E D^*(P) \le 60.9984 \sqrt{dm^d} \frac{\sqrt{\ln(4em/d)}+2.9599}{\sqrt{m/d}}.\] 
\end{theorem}


The main reason why the upper bound proof of Pausinger and Steinerberger~\cite{PausingerS16} does not give the right order of magnitude is the following. Using a similar reduction to one-dimensional discrepancies as in their lower bound proof, Pausinger and Steinerberger again are able to give a strong bound (including an exponentially decreasing tail) for the maximum discrepancy among all boxes with upper right corner lying in the same $m$-cube. To obtain an upper bound valid for all boxes, a union bound is employed. Such a  union bound, naturally, ignores any positive correlation between the discrepancies of boxes with corner point in different, but close-by $m$-cubes. So from a broader perspective, the proof again does a very precise analysis inside the $m$-cubes, but ignores the overall combinatorial structure of the problem. 

Two arguments have been used in the past to better exploit the positive correlation of similar boxes. Heinrich et al.~\cite{HeinrichNWW01} used deep results of Talagrand and Haussler from the theory of empirical processes to give the first proof of the $O(\sqrt{dN})$ discrepancy bound for $N$ independent uniformly distributed random points in $[0,1)^d$. With a non-trivial, purely combinatorial decomposition argument called \emph{dyadic chaining}, Aistleitner~\cite{Aistleitner11} and later Gnewuch and Hebbinghaus~\cite{GnewuchH21} reproved this bound and gave explicit (and small) values for the leading constant

In the following, we show that Aistleitner's {dyadic chaining} technique can also be used for the non-uniformly distributed point sets stemming from jittered sampling. The main difference, and reason for the stronger discrepancy bound, comes from noting that the grid points $\Gamma_0 = \{0, \frac1m, \frac 2m, \dots, 1\}^d$ form a $\frac dm$-cover (see below for a definition) such that each grid point $x \in \Gamma_0$ defines a rectangle $[0,x)$ with discrepancy $0$. Consequently, we can start the dyadic chaining construction with these grid points as coarsest cover. Note that this is not a very efficient (that is, small) cover, but this has no influence on the overall efficiency of the construction as one can verify from the proof below. 

The stochastic dependencies in the jittered sampling random experiment, interestingly, impose no additional difficulties. Still the number of points in an arbitrary measurable set can be written as sum of independent $0,1$ random variables. Unlike in the proof of~\cite{Aistleitner11}, these are not identically distributed, but this has no influence on the applicability of most Chernoff-type large deviation bounds.

Of course, in addition to these observations, it remains to estimate the expected star discrepancy in an analogous way as done in~\cite{Aistleitner11}, which requires some care. As for the lower bound in the previous section, in this article we do not take great care for obtaining a small leading constant, thus simplifying some calculations as compared to~\cite{Aistleitner11} and~\cite{GnewuchH21}.	The reader familiar with Aistleitner's proof will also note that we prefer to work with half-open rectangles, but clearly this makes no difference.

We recall the definitions of $\delta$-covers and $\delta$-bracketing covers. We use the notation  $\overline{[x,y)} := [0,y) \setminus [0,x)$ for all $x, y \in [0,1)^d$ with $x \le y$. Let $\delta>0$. A set $\Gamma \subseteq [0,1]^d$ is called \emph{$\delta$-cover} if for each $y \in [0,1)^d$ there are $x,z \in \Gamma \cup \{0\}$ such that $x \le y \le z$ and $\lambda(\overline{[x,z)}) \le \delta$. In particular, and this is what we will need only, we have $\lambda(\overline{[x,y)}) \le \delta$. A set $\Delta \subseteq ([0,1]^d)^2$ is called \emph{$\delta$-bracketing cover} if for each $x \in [0,1)^d$ there is a pair $(v^x,w^x) \in \Delta$ such that $v^x \le x \le w^x$ and $\lambda(\overline{[v^x,w^x)}) \le \delta$. 

\begin{proof}[Proof of Theorem~\ref{thm:upper}]
  We start by defining a sequence of $\delta$-covers of increasing precision. Let $\delta_0 := \frac dm$ and $\Gamma_0 :=  \{0,\frac 1m,  \frac 2m, \dots, 1\}^d$. Then $\Gamma_0$ is a $\delta_0$-cover. Let $K := \lfloor \frac{d-1}{2} \log_2 m \rfloor$. For $i = 1, \dots, K-1$, let $\delta_i := 2^{-i} \frac dm$ and let $\Gamma_i$ be a $\delta_i$-cover with $|\Gamma_i| \le (4e/\delta_i)^d = (4e 2^i \frac md)^d =: \gamma_i$. Finally, let $\delta_K := 2^{-K} \frac dm$ and let $\Delta_K$ be a $\delta_K$-bracketing cover with $|\Delta_K| \le (4e/\delta_K)^d  =: \gamma_K$. Such covers exist by Theorem 1.15 of~\cite{Gnewuch08}, see also Lemma~1 of~\cite{Aistleitner11}. 
  
  By definition of bracketing covers, for each $x \in [0,1)^d$ there is a pair $(v_K^x,w_K^x) \in \Delta_K$ such that $v_K^x \le x \le w_K^x$ and $\lambda(\overline{[v_K^x,w_K^x)}) \le \delta_K$. By elementary properties of the discrepancy function, we have \[|\disc([0,x))| \le \max\big\{|\disc([0,v_K^x))|, |\disc([0,w_K^x))|\big\} + N \delta_K.\] Consequently, 
  \begin{align}
  D^*(P) &\le N \delta_K + \max\big\{|\disc([0,v_K^x))|, |\disc([0,w_K^x))| \,\big|\, x \in [0,1)^d\big\}.\label{eq:disc1}
  \end{align}
  
  Note that $N \delta_K \le \frac{\sqrt{dN}}{\sqrt{m/d}}$ is of asymptotic order not larger than the bound we aim at. Consequently, it suffices in the following to analyze 
	\[
	\max\left\{|\disc([0,v_K^x))|, |\disc([0,w_K^x))| \middle| x \in [0,1)^d\right\}.
	\]
  
  To this aim, note that for each $i \in [0..K-1]$ and each $x \in [0,1)^d$ there is a $v_i(x) \in \Gamma_i \cup \{0\}$ such that $v_i(x) 
  \le x$ and $\lambda(\overline{[v_i(x),x)}) \le \delta_i$. 
  
  For each $x \in [0,1)^d$ we define $p_{K+1}^x := w_K^x$, $p_K^x := v_K^x$, and recursively for $i = K-1, \dots, 0$, we define $p_i^x := v_i(p_{i+1}^x)$. By construction, the sets $B_i^x := \overline{[p_{i-1}^x,p_{i}^x)}$, $i = 1, \dots, K+1$, are disjoint. By the additivity of $\disc(\cdot)$ and thus the subadditivity of $|\disc(\cdot)|$, we obtain 
  \begin{align}
  |\disc([0,v_K^x))| &\le \sum_{i = 1}^{K} |\disc(B_i^x)|, \label{eq:disc2a}\\
  |\disc([0,w_K^x))| &\le \sum_{i = 1}^{K+1} |\disc(B_i^x)| \label{eq:disc2b}.
  \end{align}
    
  For this reason, we now proceed by analyzing the discrepancies of the sets $B_i^x$. To this aim, keep in mind that (i)~we have $\lambda(B_i^x) \le \delta_{i-1}$ for all $i$ and $x$, and that (ii)~for fixed $i$, the number of different $B_i^x \neq \emptyset$ is at most $\gamma_i = (4e2^i\frac md)^d$, for $i \in [K]$, and $\gamma_{K+1} := \gamma_K$ when $i = K+1$. 
  
  Let us first regard an arbitrary measurable set $S \subseteq [0,1)^d$. Let $\QQ$ be the set of all elementary cubes $\prod_{j=1}^d [\frac{q_j-1}{m},\frac{q_j}{m})$, $q_1, \dots, q_d \in [m]$. For each $Q \in \QQ$, we have \[\E[|P \cap S \cap Q|] = \Pr[P \cap S \cap Q \neq \emptyset] = \frac{\lambda(S \cap Q)}{\lambda(Q)} = N \lambda(S \cap Q).\] 
  Consequently, we may write 
  \begin{align*}
  \disc(S) &= |P \cap S| - N \lambda(S) \\
  &= \sum_{Q \in \QQ} \big(|P \cap S \cap Q| - N \lambda(S \cap Q)\big)\\
  &= \sum_{Q \in \QQ} \big(|P \cap S \cap Q| - \E[|P \cap S \cap Q|]\big)
  \end{align*}
  as a sum of $N$ independent random variables $Z_Q := |P \cap S \cap Q| - \E[|P \cap S \cap Q|]$, each taking values in the interval $[-1,1]$, each having expectation $\E[Z_Q]=0$, and each having variance $\Var[Z_Q] = \Var[|P \cap S \cap Q|] \le \E[|P \cap S \cap Q|] = N \lambda(S \cap Q)$. Consequently, Bernstein's inequality gives \[\Pr[|\disc(S)| \ge t] \le 2 \exp\bigg(-\frac{t^2}{2 \sum_Q \Var[Z_Q] + 2t/3}\bigg) \le 2 \exp\bigg(-\frac{t^2}{2N\lambda(S) + 2t/3}\bigg).\]

  For all $i \in [K+1]$ and $\ell \in \R_{\ge 1}$ let $t_{i,\ell} := 2 C \ell d \sqrt{\frac{N}{m} \frac{\ln(2^{2i} 4e (m/d))}{2^{i-1}}}$ for some $C \ge 1$. Let $x \in [0,1)^d$ and $i \in [K+1]$. By the above, we have 
  \begin{align*}
  \Pr[|\disc&(B_i^x)| \ge t_{i,\ell}] \\
  &\le 2 \exp\bigg(-\frac{t_{i,\ell}^2}{2N\delta_{i-1} + 2t_{i,\ell}/3}\bigg)\\
  &\le 2 \exp\bigg(-\frac{t_{i,\ell}^2}{2 \max\{2N\delta_{i-1},2t_{i,\ell}/3\}}\bigg)\\
  &\le 2 \max\bigg\{\exp\bigg(-\frac{t_{i,\ell}^2}{4 N\delta_{i-1}}\bigg), \exp\bigg(-\frac{t_{i,\ell}^2}{4 t_{i,\ell}/3}\bigg)\bigg\}\\
  &\le 2 \bigg(\exp\bigg(-\frac{t_{i,\ell}^2}{4 N\delta_{i-1}}\bigg) + \exp\bigg(-\frac{t_{i,\ell}^2}{4 t_{i,\ell}/3}\bigg)\bigg)\\
  &= 2 \exp\bigg(-\frac{4 \frac{C^2 d^2 \ell^2 N}{m} 2^{-i+1} \ln(2^{2i} 4e (m/d))}{4N 2^{-i+1} (d/m)}\bigg) \\
  & \quad + 2 \exp\bigg(-\tfrac 32 C \ell d \sqrt{\frac{N}{m} \frac{\ln(2^{2i} 4e (m/d))}{2^{i-1}}}\,\bigg)\\
  &\le 2 \exp(-d \ell^2 \ln(2^{2i} 4e (m/d))) + 2 \exp\bigg(-\tfrac 32 C  \sqrt{\frac Nm \frac{\ln(2^{2i})}{m^{(d-1)/2}}}\,\bigg)^{d\ell} \\
  &= 2 (2^{2i} 4e (m/d))^{-d \ell^2} + 2 \exp\big(\tfrac 3 {12} C  m^{(d-1)/4} \sqrt{2i \ln(2)}\,\big)^{-6d\ell} \\
  &\le 2 (2^{2i} 4e (m/d))^{-d \ell^2} + 2 \big(\tfrac 14 C  m^{(d-1)/4} \sqrt{2i \ln(2)}\,\big)^{-6d\ell} =: q_{i,\ell}.
  \end{align*}

  With a simple union bound, choosing $C = 8^{1/12} (e/4)^{1/6} 4/\sqrt{\ln(2)} = 2^{23/12} e^{1/6} / \sqrt{\ln(2)} \le 5.3573$, 
and recalling that $d \ge 2$ and $\ell \ge 1$, we estimate
  \begin{align*}
  \Pr[\exists x \in &[0,1)^d : |\disc(B_i^x)| \ge t_{i,\ell}] \le \gamma_i q_{i,\ell}\\
  & \le (4e 2^{i} \tfrac md)^d  \cdot 2 (2^{2i} 4e \tfrac md)^{-d\ell^2} \\
  & \quad + (4e m^{(d-1)/2} \tfrac md)^d \cdot 2 \big(\tfrac 14 C m^{(d-1)/4} \sqrt{2i \ln(2)}\,\big)^{-6d\ell}\\
  & \le 2 \cdot 2^{-id\ell^2} +  2 \bigg(m^{-d+2} i^{-3} \frac{4e 4^6}{2 C^6 2^3 \ln(2)^3}  \bigg)^{d\ell}\\
  & \le 2 \cdot 2^{-id\ell^2} +  2 \big(\sqrt{1/8} \, m^{-d+2}  i^{-3} \big)^{d\ell}\\
  & \le 2 \cdot 4^{-i\ell^2} +  2 \cdot 8^{-\ell} m^{-d+2}  i^{-6} \le 2(4^{-i\ell} + 8^{-\ell}i^{-6}).
  \end{align*}
  Another union bound together with $\ell \ge 1$ and $\zeta(6) = \pi^6 / 945 \le 1.0174$ gives
  \begin{align*}  
  \Pr[\exists i \in &[K+1] \exists x \in [0,1)^d : |\disc(B_i^x)| \ge t_{i,\ell}] \\
  & \le \sum_{i = 1}^{K+1} 2(4^{-i\ell} + 8^{-\ell}i^{-6})\\
  & \le 2 \cdot 4^{-\ell} \sum_{i = 0}^\infty 4^{-i} + 2 \cdot 8^{-\ell} \sum_{i = 1}^\infty i^{-6}\\
  & = 4^{-\ell} \tfrac 8 3 + 2 \cdot 8^{-\ell} \zeta(6) \le  4^{-\ell} (\tfrac 83+\zeta(6)) < 3.7 \cdot 4^{-\ell}.
  \end{align*}  
  We compute 
  \begin{align*}
  \sum_{i=1}^{K+1} t_{i,\ell} & \le \sum_{i = 1}^\infty 2 C \ell d \sqrt{\frac{N}{m} \frac{\ln(2^{2i} 4e (m/d))}{2^{i-1}}}\\
  & \le 2 C \ell d \sqrt{\frac{N}{m}} \bigg( \sum_{i=1}^\infty \sqrt{\frac{2i \ln(2)}{2^{i-1}}} + \sum_{i=1}^\infty \sqrt{\frac{\ln(4em/d)}{2^{i-1}}}\,\bigg)\\
  & \le 2 C \ell d \sqrt{\frac{N}{m}} \bigg(5(1+\sqrt 2 / 2) \sqrt{2 \ln 2}  + \sqrt{\ln(4em/d)} \frac{\sqrt 2}{\sqrt 2 - 1}\bigg) =: \ell D
  \end{align*}
  using the estimate $\sum_{i=1}^\infty \sqrt{i/2^{i-1}}  \le 1 + \sum_{i=2}^\infty i/2^{i/2} = 1 + \sum_{i=1}^\infty (2i)/2^{i} + \sum_{i=1}^\infty (2i+1)/2^{i+1/2} = 1 + 2 \sum_{i=1}^\infty i/2^i + \sqrt 2 \sum_{i=1}^\infty i/2^i + 2^{-1/2} \sum_{i=1}^\infty 2^{-i} = 1 + 4 + 2\sqrt 2 + 2^{-1/2} = 5(1+\sqrt 2 / 2)$.
  
  Consequently, by~\eqref{eq:disc1},~\eqref{eq:disc2a}, and~\eqref{eq:disc2b}, we have $\Pr[D^*(P) \ge \ell D+N\delta_K] \le 3.7 \cdot 4^{-\ell}$ for all $\ell \ge 1$. From this we easily derive the following bound on the expectation of $D^*(P)$.
  \begin{align*}
  \E D^*(P) &= \int_0^\infty \Pr[D^*(P) \ge x] dx\\
  &\le D + N\delta_K + D \int_1^\infty \Pr[D^*(P) \ge \ell D + N\delta_K] d\ell\\
  &\le D (1 + 3.7 \cdot 4^{-1} / \ln(4)) +N\delta_K \le 1.6673 D +N\delta_K\\ 
  &\le 1.6673 \cdot 2 C d \sqrt{\frac{N}{m}} \bigg(5 (1+\sqrt 2/2) \sqrt{2 \ln 2}  + \frac{\sqrt{2 \ln(4em/d)}}{\sqrt 2 - 1}\bigg) + \sqrt{\frac{dN}{m/d}} \\
  &\le 17.8645 \sqrt{\frac{dN}{m/d}} \big(10.0499 + 3.4143 \sqrt{\ln(4em/d)} + 1/17.8645\big)\\
  &\le 60.9948 \sqrt{\frac{dN}{m/d}} \big(\sqrt{\ln(4em/d)} + 2.9599\big).\qedhere
  \end{align*}  
\end{proof}

\section{Small Numbers of Points}\label{sec:small}

The results proven in the previous sections indicate that jittered sampling will not lead to discrepancies of asymptotic order smaller than those of independent random points when $m = O(d)$. We shall not prove this here, since an adaptation of our lower bound proof above to small numbers of points or an adaptation of the proof in~\cite{Doerr14} to jittered sampling both appear not straight-forward, and since at the same time it is very hard to imagine that jittered sampling has an asymptotic advantage in some regime $m = o(d)$ given that we have disproven such an advantage for $m=d$.

With no asymptotic advantage expected for $m = O(d)$, one could think that jittered sampling is interesting only for numbers $N$ of points that are super-exponential in $d$, namely at least of the order $\omega(d)^d$. We now argue that this might be a too pessimistic point of view, simply because possible constant factor improvements or even advantages in lower-order terms can be interesting as well. We note that Pausinger and Steinerberger~\cite{PausingerS16} conducted also a small experimental analysis and observed, e.g., that the (empirical) expected discrepancy of $N = 3^5$ points in dimension $d = 5$ (that is, we have $m=3$ in the jittered sampling case) is $0.1200$ for independent random points and $0.1046$ for jittered sampling. Such advantages may not appear large, but given that they come for free (generating a jittered sampling is not more costly that generating an independent sample), they are interesting. In addition, it is known that such point sets satisfy the known discrepancy guarantees for independent points sets (we discuss this in more detail below), so also in terms of proven quality guarantees switching from independent random points to jittered sampling gives no disadvantages.

Since we need $m \ge 2$ in our definition of jittered sampling, we necessarily have a number of points that is at least exponential in $d$. We argue now that also for smaller numbers of points some forms of jittering can be defined which promise to be at least as good as independent points, with lower-order advantages likely to appear. Assume that $N$ is a power of two, but less than $2^d$, the smallest number for which our jittered sampling was defined. Hence $N = 2^{d'}$ for some $d' < d$. We define a variant of jittered sampling of $N$ points as follow. Consider the family $B_x = \prod_{i=1}^{d'} [x_i,x_i+\tfrac 12) \times [0,1)^{d-d'}, x \in \{0,\tfrac 12\}^{d'}$ of boxes and define the point set $P$ by choosing from each such box independently a point uniformly distributed in this box. 
This set of points is again easy to construct and is also known to satisfy the known discrepancy guarantees of independent random points. 

\subsection{Discrepancy Guarantees from the Result of Gnewuch and Hebbinghaus}

We now give the promised details on why the strongest currently known discrepancy guarantee for a broad class of random point sets by Gnewuch and Hebbinghaus~\cite{GnewuchH21} (which imply those of Aisleitner~\cite{Aistleitner11} and Aistleitner and Hofer~\cite{AistleitnerH14}) also applies to the jittered samplings proposed above. All this can be found in~\cite{WnukGH20}. We start by describing the discrepancy result of~\cite{GnewuchH21}.

Let $\calD_0 := \{[0,b) \setminus [0,a) \mid a, b \in [0,1)^d, a \le b\}$ and $\gamma \ge 1$. A collection $X_1, \dots, X_N$ of random points is called $\gamma$-negatively dependent (with respect to $\calD_0$) if for all $D \in \calD_0$ and all $J \subseteq [N]$ we have
\begin{align*}
  \Pr[\forall j \in J : X_j \in D] &\le \gamma \prod_{j \in J} \Pr[X_j \in D],\\
  \Pr[\forall j \in J : X_j \notin D] &\le \gamma \prod_{j \in J} \Pr[X_j \notin D].
\end{align*}
Gnewuch and Hebbinghaus~\cite[Theorem~4.4]{GnewuchH21} have shown that for all $c > 0$ and $\rho \ge 0$ the discrepancy of an $e^{\rho d}$-negatively dependent collection of $N$ points in $[0,1)^d$, each of which is uniformly distributed in $[0,1)^d$, is at most $c \sqrt{dN}$ with probability at least $1 - e^{-(1.6741 c^2 - 10.7042 - \rho)d}$. Since independent uniformly distributed random points are $1$-negatively dependent by definition, this reproves the known $O(\sqrt{dN})$ upper bound for the expected discrepancy of uniformly distributed random point sets, however with better absolute bounds such as the result that with positive probability the discrepancy is at most $2.5278 \sqrt{dN}$.

This result cannot be immediately applied to jittered sampling since there the individual points are not uniformly distributed in $[0,1)^d$. The solution proposed in~\cite{WnukGH20} is to regard the symmetrization of the jittered sampling set. Let $B_1, \dots, B_N$ be a partition of $[0,1)^d$ into sets of equal Lebesgue measure $1/N$. Let $X_1, \dots, X_N$ be a jittered sampling point set with respect to these, that is, for each $j \in [N]$ the point $X_j$ is chosen uniformly from $B_j$. Let $\sigma : [N] \to [N]$ be a permutation chosen uniformly at random. Let $X'_i = X_{\sigma(i)}$ for all $i \in [N]$. Then each $X'_i$ is uniformly distributed in $[0,1)^d$. Unfortunately, the $X'_i$ are not independent anymore (if $X'_1 \in B_1$, for example, then $X'_2$ cannot be in $B_1$; hence $\Pr[X'_2 \in B_1 \mid X'_1 \in B_1] = 0 \neq \frac 1N = \Pr[X'_2 \in B_1]$). They are, however, $1$-negatively dependent as shown in~\cite[Theorem~4.4]{WnukGH20}. Hence the discrepancy estimate of Gnewuch and Hebbinghaus applies to the $X'_i$ and gives the same guarantees as for independent points. Since, viewed as multiset, the $X_i$ and the $X'_i$ are identically distributed, this discrepancy also holds for the jittered sampling $(X_i)$. 

\subsection{Extending Aistleitner's Upper Bound to Jittered Sampling}

Since the approach described in the previous subsection is slightly technical (we lose the independence of the points due to the symmetrization and then cope with this by establishing a negative dependence property and using a discrepancy guarantee for negatively dependent point sets), we now show that the earlier (and weaker in terms of constants) discrepancy guarantees of Aistleitner~\cite{Aistleitner11} and Aistleitner and Hofer~\cite{AistleitnerH14} quite directly apply also to jittered samplings. Since this observation might be useful also for extending future discrepancy estimates for independent uniformly distributed points to jittered sampling, we give some more details.

An easy inspection of Aistleitner's~\cite{Aistleitner11} proof reveals that the only way this upper bound analysis of the discrepancy of a uniformly distributed random point set relies on the randomness of the points is by applying Chernoff-type concentration bounds to the expected number of points in a measurable subset $I$ of $[0,1)^d$. If $X_1, \dots, X_N$ are independently and uniformly distributed in $[0,1)^d$, then the indicator random variables $\eins_I(X_j)$ for the events ``$X_j \in I$'', $j \in [N]$, are independent binary random variables with success probability $p_{I,j} = \lambda(I)$. Letting, for all $j \in [N]$, $Z_{I,j} = \eins_I(X_j) - \E \eins_I(X_j) = \eins_I(X_j) - p_{I,j} = \eins_I(X_j) - \lambda(I)$ denote the centralization of $\eins_I(X_j)$, we have $\disc_P(I) = |P \cap I| - N \lambda(I) =  \sum_{j = 1}^N Z_{I,j} =: Z_I$, that is, the discrepancy of $I$ can be written as the sum of the independent random variables $Z_{I,j}$ (having expectation zero and taking values in an interval of length~$1$) and can thus be estimated via strong concentration bounds. Here Aistleitner (see the lower half of page~536 of~\cite{Aistleitner11}) only uses the two estimates
\begin{align*}
\Pr[|Z_I|\ge t] & \le 2\exp\left(-\frac{t^2}{2 \Var[Z_I] + \frac 23 t}\right) = 2 \exp\left(-\frac{t^2}{2 N \lambda(I) (1-\lambda(I)) + \frac 23 t}\right),\\ 
\Pr[|Z_I| \ge t] & \le 2 \exp\left(-\frac{2t^2}{N}\right).
\end{align*}

We now argue that the same estimates hold for the corresponding expressions for jittered sampling, which shows that Aistleitner's proof also applies to jittered sampling. Since it takes no additional effort, we show this result for a more general type of random point sets. Let $\mu_1, \dots, \mu_N$ be probability measures on $[0,1)^d$ (equipped with the Lebesgue $\sigma$-algebra) such that for each $\lambda$-measurable set $I \subseteq [0,1)^d$ we have $\sum_{j=1}^N \mu_j(I) = N \lambda(I)$. Let $P'$ be a random point set obtained by sampling, for each $j \in [N]$ independently, a point $X'_j$ according to $\mu_j$. We observe that this construction includes both independent uniformly distributed points and jittered sampling (but not the generalized jittered samplings of~\cite{WnukGH20}).

We now show that such a point set satisfies the discrepancy guarantees for independent uniformly distributed points given by Aistleitner~\cite{Aistleitner11}. Let $I$ be a measurable subset of $[0,1)^d$. By construction, the indicator random variables $\eins_I(X'_j)$, $j \in [N]$, are again independent, but now $p'_{I,j} = \mu_j(I)$. Let $Z'_{I,j} := \eins_I(X'_j) - \E \eins_I(X'_j) = \eins_I(X'_j) - p'_{I,j} = \eins_I(X'_j) - \mu_j(I)$. Then again $\disc_{P'}(I) = |P' \cap I| - N \lambda(I) =  \sum_{j = 1}^N Z'_{I,j} =: Z'_I$. Since $Z'_I$ is a sum of independent random variables (each with expectation zero and taking values in an interval of length $1$), we again have
\begin{align*}
\Pr[|Z'_I|\ge t] & \le 2\exp\left(-\frac{t^2}{2 \Var[Z'_I] + \frac 23 t}\right) \le 2 \exp\left(-\frac{t^2}{2 N \lambda(I) (1-\lambda(I)) + \frac 23 t}\right)\\ 
\Pr[|Z'_I| \ge t] & \le 2 \exp\left(-\frac{2t^2}{N}\right).
\end{align*}
For the second inequality in the first line, note that $\Var[Z'_I] = \sum_{j = 1}^N p'_{I,j} (1 - p'_{I,j})  = \sum_{j = 1}^N (p'_{I,j}  - (p'_{I,j})^2)$. We have $\sum_{j=1}^N p'_{I,j} = \sum_{j=1}^N \mu_j(I) = N \lambda(I)$ by construction. Knowing that a sum of squares of $n$ variables with fixed sum $S$ is minimal for all variables being equal to $S/n$, that is, $\min\{x_1^2 + \dots + x_n^2 \mid x_1, \dots, x_n \in \R, x_1 + \dots + x_n = S\} = n \cdot (S/n)^2$, we obtain $\Var[Z'_I] \le N \lambda(I) - N \lambda(I)^2 = N \lambda(I) (1 - \lambda(I))$. Consequently, we have the same concentration bounds for the discrepancy of the set $I$ as when using independent sample points. Since Aistleitner's proof refers to the sample set only via these two estimates, his proof applies equally well to our generalized point set construction.

The same argument immediately also extends the upper bounds of Aistleitner and Hofer~\cite{AistleitnerH14} to our point sets. We are optimistic that this argument can be used for future discrepancy estimates shown roughly along these lines as well. We note that~\cite{GnewuchH21} also uses only the two concentration bounds discussed above (except that it is argued why these also extend to $\gamma$-negatively dependent random variables). Hence the only reason why the result of~\cite{GnewuchH21} does not immediately extend to the random samplings defined above is the requirement that each point is uniformly distributed in $[0,1)^d$.

\section{Conclusion}

In this work, we determined (apart from constant factors) the expected star discrepancy of a jittered sampling point set having $m^d$ points, $m \ge d$. This improves both the previous upper and lower bound and also removes the unquantified restriction that ``$m$ is sufficiently large compared to $d$''. These improvements are made possible by exhibiting suitable combinatorial structures in the problem which then allow an easier analysis via discrete probability theory. We feel that this might, also for other discrepancy problems on random point sets, be an approach more suitable than trying to describe the discrepancy directly via continuous-domain random variables.

From our main result and the known result that the known discrepancy guarantees for independent random points extend to jittered sampling (in particular for $m < d$, where we have given no guarantee), we see no reason to not generally recommend to use jittered sampling to obtain low-discrepancy point sets rather than independent random points

\subsection*{Acknowledgments}

This work was triggered by the 2nd International Workshop on Discrepancy Theory and Applications organized by Dmitriy Bilyk, Luca Brandolini, William Chen, Anand Srivastav, Giancarlo Travaglini. The author would like to thank the organizers for inviting him to this fruitful meeting. This work also profited from discussions with Michael Gnewuch and Marcin Wnuk, which are gratefully acknowledged.


\begin{thebibliography}{HNWW01}

\bibitem[AH14]{AistleitnerH14}
Christoph Aistleitner and Markus Hofer.
\newblock Probabilistic discrepancy bound for {M}onte {C}arlo point sets.
\newblock {\em Mathematics of Computation}, 83:1373--1381, 2014.

\bibitem[Ais11]{Aistleitner11}
Christoph Aistleitner.
\newblock Covering numbers, dyadic chaining and discrepancy.
\newblock {\em Journal of Complexity}, 27:531--540, 2011.

\bibitem[ASYM16]{AvronSYM16}
Haim Avron, Vikas Sindhwani, Jiyan Yang, and Michael~W. Mahoney.
\newblock Quasi-{M}onte {C}arlo feature maps for shift-invariant kernels.
\newblock {\em Journal of Machine Learning Research}, 17:120:1--120:38, 2016.

\bibitem[Bec87]{Beck87}
J\'{o}zsef Beck.
\newblock Irregularities of distribution. {I}.
\newblock {\em Acta Mathematica}, 159:1--49, 1987.

\bibitem[Bel81]{Bellhouse81}
David~R. Bellhouse.
\newblock Area estimation by point-counting techniques.
\newblock {\em Biometrics}, 37:303--312, 1981.

\bibitem[BL13]{BilykL13}
Dmitriy Bilyk and Michael Lacey.
\newblock The supremum norm of the discrepancy function: recent results and
  connections.
\newblock In {\em Monte Carlo and Quasi-Monte Carlo Methods 2012}, pages
  23--38. Springer, 2013.

\bibitem[Cha00]{Chazelle00}
Bernard Chazelle.
\newblock {\em The {D}iscrepancy Method}.
\newblock Cambridge University Press, 2000.

\bibitem[CPC84]{CookPC84}
Robert~L. Cook, Thomas~K. Porter, and Loren~C. Carpenter.
\newblock Distributed ray tracing.
\newblock In {\em Conference on Computer Graphics and Interactive Techniques,
  {SIGGRAPH} 1984}, pages 137--145. {ACM}, 1984.

\bibitem[DDG18]{DoerrDG18}
Benjamin Doerr, Carola Doerr, and Michael Gnewuch.
\newblock Probabilistic lower bounds for the discrepancy of {L}atin hypercube
  samples.
\newblock In Josef Dick, Frances~Y. Kuo, and Henryk Wo{\'{z}}niakowski,
  editors, {\em Contemporary Computational Mathematics - A Celebration of the
  80th Birthday of Ian Sloan}, pages 339--350. Springer, 2018.

\bibitem[DGS04]{DoerrGS04}
Benjamin Doerr, Michael Gnewuch, and Anand Srivastav.
\newblock Bounds and construction for the star discrepancy via $\delta$-covers.
\newblock Berichtsreihe des Mathematischen Seminars der
  Christian-Albrechts-Universit\"at zu Kiel, Report 04-13, 2004.

\bibitem[DGS05]{DoerrGS05}
Benjamin Doerr, Michael Gnewuch, and Anand Srivastav.
\newblock Bounds and constructions for the star-discrepancy via
  $\delta$-covers.
\newblock {\em Journal of Complexity}, 21:691--709, 2005.

\bibitem[Doe14]{Doerr14}
Benjamin Doerr.
\newblock A lower bound for the discrepancy of a random point set.
\newblock {\em Journal of Complexity}, 30:16--20, 2014.

\bibitem[DP10]{DickP10}
Josef Dick and Friedrich Pillichshammer.
\newblock {\em Digital Nets and Sequences}.
\newblock Cambridge University Press, 2010.

\bibitem[DT97]{DrmotaT97}
Michael Drmota and Robert~F. Tichy.
\newblock {\em Sequences, Discrepancies and Applications}, volume 1651 of {\em
  Lecture Notes in Mathematics}.
\newblock Springer, 1997.

\bibitem[Fel71]{Feller71}
William Feller.
\newblock {\em Introduction to Probability Theory and Its Applications},
  volume~2.
\newblock Wiley, 2nd edition, 1971.

\bibitem[FW93]{FangW93}
Kai-Tai Fang and Yuan Wang.
\newblock {\em Number-Theoretic Methods in Statistics}.
\newblock CRC Press, 1993.

\bibitem[GH21]{GnewuchH21}
Michael Gnewuch and Nils Hebbinghaus.
\newblock Discrepancy bounds for a class of negatively dependent random points
  including {L}atin hypercube samples.
\newblock {\em Annals of Applied Probability}, 31:1944--1965, 2021.

\bibitem[Gne08]{Gnewuch08}
Michael Gnewuch.
\newblock Bracketing numbers for axis-parallel boxes and applications to
  geometric discrepancy.
\newblock {\em Journal of Complexity}, 24:154--172, 2008.

\bibitem[Gne12]{Gnewuch12}
Michael Gnewuch.
\newblock Entropy, randomization, derandomization, and discrepancy.
\newblock In Leszek Plaskota and Henryk Wo\'zniakowski, editors, {\em Monte
  Carlo and Quasi-Monte Carlo Methods 2010}, volume~23 of {\em Proceedings in
  Mathematics and Statistics}, pages 43--78. Springer, 2012.

\bibitem[Hla61]{Hlawka61}
Edmund Hlawka.
\newblock {F}unktionen von beschr\"ankter {V}ariation in der {T}heorie der
  {G}leich\-verteilung.
\newblock {\em Annali di Matematica Pura ed Applicata}, 54:325--333, 1961.

\bibitem[HNWW01]{HeinrichNWW01}
Stefan Heinrich, Erich Novak, Grzegorz~W. Wasilkowski, and Henryk
  Wo\'{z}niakowski.
\newblock The inverse of the star-discrepancy depends linearly on the
  dimension.
\newblock {\em Acta Arithmetica}, 96:279--302, 2001.

\bibitem[KM05]{KimuraM05}
Shuhei Kimura and Koki Matsumura.
\newblock Genetic algorithms using low-discrepancy sequences.
\newblock In {\em Genetic and Evolutionary Computation Conference, {GECCO}
  2005}, pages 1341--1346. {ACM}, 2005.

\bibitem[Kok43]{Koksma42}
Jurjen~F. Koksma.
\newblock A general theorem from the uniform distribution modulo 1 (in
  {D}utch).
\newblock {\em Mathematica B (Zutphen)}, 1:7--11, 1942/43.

\bibitem[Mat99]{Matousek99}
Ji\v{r}\'i Matou{\v{s}}ek.
\newblock {\em Geometric {D}iscrepancy}.
\newblock Springer-Verlag, Berlin, 1999.

\bibitem[MBC79]{McKayBC79}
Michael~D. McKay, Richard~J. Beckman, and William~J. Conover.
\newblock Comparison of three methods for selecting values of input variables
  in the analysis of output from a computer code.
\newblock {\em Technometrics}, 21:239--245, 1979.

\bibitem[Nie92]{Niederreiter92}
Harald Niederreiter.
\newblock {\em Random Number Generation and Quasi-{M}onte {C}arlo Methods},
  volume~63 of {\em CBMS-NSF Regional Conference Series in Applied
  Mathematics}.
\newblock Society for Industrial and Applied Mathematics (SIAM), Philadelphia,
  PA, 1992.

\bibitem[OT05]{OwenT05}
Art~B. Owen and Seth~D. Tribble.
\newblock A quasi-{M}onte {C}arlo {M}etropolis algorithm.
\newblock {\em Proceedings of the National Academy of Sciences},
  102:8844--8849, 2005.

\bibitem[Owe03]{Owen03}
Art~B. Owen.
\newblock Quasi-monte carlo sampling.
\newblock In {\em Monte Carlo Ray Tracing: ACM SIGGRAPH 2003 Course Notes 44},
  pages 69--88. ACM, 2003.

\bibitem[PS16]{PausingerS16}
Florian Pausinger and Stefan Steinerberger.
\newblock On the discrepancy of jittered sampling.
\newblock {\em Journal of Complexity}, 33:199--216, 2016.

\bibitem[Rob55]{Robbins55}
Herbert Robbins.
\newblock A remark on {S}tirling's formula.
\newblock {\em The American Mathematical Monthly}, 62:26--29, 1955.

\bibitem[TG07]{TeytaudG07}
Olivier Teytaud and Sylvain Gelly.
\newblock {DCMA:} yet another derandomization in covariance-matrix-adaptation.
\newblock In {\em Genetic and Evolutionary Computation Conference, {GECCO}
  2007}, pages 955--963. {ACM}, 2007.

\bibitem[WGH20]{WnukGH20}
Marcin Wnuk, Michael Gnewuch, and Nils Hebbinghaus.
\newblock On negatively dependent sampling schemes, variance reduction, and
  probabilistic upper discrepancy bounds.
\newblock In Dmitriy Bilyk, Josef Dick, and Friedrich Pillichshammer, editors,
  {\em Discrepancy Theory}, volume~26 of {\em Radon Series on Computational and
  Applied Mathematics}, pages 43--67. De Gruyter, 2020.

\end{thebibliography}
\end{document}